\documentclass[12pt,reqno]{amsart}
\usepackage{amsmath, amsthm, amssymb, stmaryrd}
\usepackage{amsmath}
\makeatletter
\DeclareFontFamily{U}  {MnSymbolF}{}
\DeclareSymbolFont{symbolsMN}{U}{MnSymbolF}{m}{n}
\SetSymbolFont{symbolsMN}{bold}{U}{MnSymbolF}{b}{n}
\DeclareFontShape{U}{MnSymbolF}{m}{n}{
    <-6>  MnSymbolF5
   <6-7>  MnSymbolF6
   <7-8>  MnSymbolF7
   <8-9>  MnSymbolF8
   <9-10> MnSymbolF9
  <10-12> MnSymbolF10
  <12->   MnSymbolF12}{}
\DeclareFontShape{U}{MnSymbolF}{b}{n}{
    <-6>  MnSymbolF-Bold5
   <6-7>  MnSymbolF-Bold6
   <7-8>  MnSymbolF-Bold7
   <8-9>  MnSymbolF-Bold8
   <9-10> MnSymbolF-Bold9
  <10-12> MnSymbolF-Bold10
  <12->   MnSymbolF-Bold12}{}
\DeclareMathSymbol{\tbigtimes}{\mathop}{symbolsMN}{2}
\newcommand*{\bigtimes}{%
  \DOTSB
  \tbigtimes
  \slimits@ 
}
\makeatother

\advance \topmargin by -\headheight
\advance \topmargin by -\headsep
     
\evensidemargin \oddsidemargin
\newtheorem{theorem}{Theorem}[section]
\newtheorem{lemma}[theorem]{Lemma}

\newtheorem{conjecture}[theorem]{Conjecture}

\theoremstyle{definition}

\theoremstyle{remark}

\numberwithin{equation}{section}

\def\bfa{{\mathbf a}}

\def\bfh{{\mathbf h}}

\def\bfn{{\mathbf n}}

\def\bft{{\mathbf t}}

\def\bfw{{\mathbf w}}
\def\bfx{{\mathbf x}}
\def\bfy{{\mathbf y}}
\def\bfz{{\mathbf z}}

\def\calI{{\mathcal I}}

 \def\Ktil{{\widetilde K}}

\def\Ktil{\widetilde K}

\def\dbN{{\mathbb N}}  
\def\dbR{{\mathbb R}}
\def\dbZ{{\mathbb Z}}

\def\grA{{\mathfrak A}}
\def\grB{{\mathfrak B}}

\def\grD{{\mathfrak D}}
\def\grf{{\mathfrak f}}\def\grF{{\mathfrak F}}

\def\grJ{{\mathfrak J}}
\def\grk{{\mathfrak k}} \def\grK{{\mathfrak K}}

\def\grm{{\mathfrak m}}\def\grM{{\mathfrak M}}
\def\grN{{\mathfrak N}}\def\grn{{\mathfrak n}}

\def\grS{{\mathfrak S}}\def\grP{{\mathfrak P}}
\def\grW{{\mathfrak W}}\def\grB{{\mathfrak B}}
\def\grK{{\mathfrak K}}\def\grp{{\mathfrak p}}

\def\grU{{\mathfrak U}}
\def\grV{{\mathfrak V}}
\def\grW{{\mathfrak W}}

\def\alp{{\alpha}} \def\bfalp{{\boldsymbol \alpha}} 
\def\bet{{\beta}}  \def\bfbet{{\boldsymbol \beta}}
\def\gam{{\gamma}} \def\Gam{{\Gamma}}
\def\bfgam{{\boldsymbol \gam}} 
\def\del{{\delta}} \def\Del{{\Delta}}

\def\tet{{\theta}} \def\bftet{{\boldsymbol \theta}} 
\def\Tet{{\Theta}}

\def\sig{{\sigma}}

\def\ome{{\omega}} \def\Ome{{\Omega}}

\def\eps{\varepsilon}

\def\le{\leqslant} \def\ge{\geqslant}

\def\d{{\,{\rm d}}}

\begin{document}
\title[Subconvexity in Vinogradov systems]{Subconvexity in inhomogeneous\\ Vinogradov 
systems}
\author[Trevor D. Wooley]{Trevor D. Wooley}
\address{Department of Mathematics, Purdue University, 150 N. University Street, West 
Lafayette, IN 47907-2067, USA}
\email{twooley@purdue.edu}
\thanks{The author's work is supported by NSF grants DMS-2001549 and DMS-1854398.}
\date{}
\dedicatory{}
\begin{abstract}When $k$ and $s$ are natural numbers and $\bfh\in \dbZ^k$, denote by 
$J_{s,k}(X;\bfh)$ the number of integral solutions of the system
\[
\sum_{i=1}^s(x_i^j-y_i^j)=h_j\quad (1\le j\le k),
\]
with $1\le x_i,y_i\le X$. When $s<k(k+1)/2$ and $(h_1,\ldots ,h_{k-1})\ne {\mathbf 0}$, 
Brandes and Hughes have shown that $J_{s,k}(X;\bfh)=o(X^s)$. In this paper we improve 
on quantitative aspects of this result, and, subject to an extension of the main conjecture in 
Vinogradov's mean value theorem, we obtain an asymptotic formula for $J_{s,k}(X;\bfh)$ in 
the critical case $s=k(k+1)/2$. The latter requires minor arc estimates going beyond 
square-root cancellation.
\end{abstract}
\maketitle

\section{Introduction} In the analysis of Diophantine systems via the Hardy-Littlewood 
(circle) method, estimates are almost always limited by the convexity barrier, the most 
optimistic bound anticipated for error terms being given by the square-root of the number 
of choices for the variables available to the system. A recent exception to this rule involves 
inhomogeneous variants of Vinogradov's mean value theorem. When $k$ and $s$ are 
natural numbers and $\bfh\in \dbZ^k$, denote by $J_{s,k}(X;\bfh)$ the number of integral 
solutions of the system
\begin{equation}\label{1.1}
\sum_{i=1}^s(x_i^j-y_i^j)=h_j\quad (1\le j\le k),
\end{equation}
with $1\le \bfx,\bfy\le X$. Then Brandes and Hughes \cite[Theorem 1]{BH2021} 
have shown that $J_{s,k}(X;\bfh)=o(X^s)$ when $s<k(k+1)/2$ and $h_j\ne 0$ for 
some index $j$ with $j\le k-1$. We emphasise that a consideration of 
diagonal solutions reveals that $J_{s,k}(X;{\mathbf 0})\gg X^s$, so one must certainly 
have $\bfh\ne {\mathbf 0}$ in order to obtain a subconvex estimate for 
$J_{s,k}(X;\bfh)$. Such estimates will also be inaccessible when 
$s\ge \tfrac{1}{2}k(k+1)+1$, since an averaging argument then confirms that there are 
numerous $k$-tuples $\bfh$ for which 
$J_{s,k}(X;\bfh)\gg X^{2s-k(k+1)/2}\gg X^{s+1}$.\par

Our goal in this paper is to sharpen the results of Brandes and Hughes both quantitatively, 
and in the range of $s$ accessible to such conclusions. We seek also to establish an 
asymptotic formula for $J_{s,k}(X;\bfh)$ in the critical case $s=k(k+1)/2$, extending to 
exponents $k>3$ our recent work \cite{Woo2021c} relevant to the cubic case, on the 
assumption of an extended version of the main conjecture in Vinogradov's mean value 
theorem. Our conclusions vary in type according to the regime of interest. We begin with 
the estimates simplest to state.

\begin{theorem}\label{theorem1.1} Suppose that $k\ge 3$ and 
$\bfh\in \dbZ^k\setminus \{{\mathbf 0}\}$. Let $l$ be the smallest index with $h_l\ne 0$. 
Then, whenever $l<k$ and $s$ is an integer with
\begin{equation}\label{1.2}
1\le s\le \tfrac{1}{2}k(k+1)-\frac{k(k+1)-l(l+1)}{2(k-l)(k-l+1)},
\end{equation}
one has
\begin{equation}\label{1.3}
J_{s,k}(X;\bfh)\ll X^{s-1/2+\eps}.
\end{equation}
In particular, this estimate holds when $1\le l\le (k+1)/3$ and $s<k(k+1)/2$.
\end{theorem}

We note that \cite[Corollary 2]{BH2021} obtains the estimate (\ref{1.3}) in the shorter 
range
\[
1\le l\le k-\tfrac{1}{2}(\sqrt{2k^2+2k+1}-1)\approx 0.29289k.
\]
The case $k=2$ is omitted from the statement of Theorem \ref{theorem1.1} because much 
stronger bounds are available in this case from the classical theory of quadratic polynomials.  
Thus, for example, the reader will have no difficulty in showing that when 
$\bfh\ne {\mathbf 0}$, one has $J_{s,2}(X;\bfh)\ll X^{s-1+\eps}$ $(s=1,2)$. Analogous 
estimates of similar strength to the latter may be obtained when $k\ge 3$ subject to 
suitable hypotheses concerning $s$ and the $k$-tuple $\bfh$.\par

\begin{theorem}\label{theorem1.2} Let $k$ be an integer with $k\ge 3$ and 
$\bfh\in \dbZ^k\setminus \{{\mathbf 0}\}$. Let $l$ be the smallest index having the 
property that $h_l\ne 0$. Then, whenever $l<k$ and $s$ is an integer with 
$1\le s\le l(l+1)/2$, one has $J_{s,k}(X;\bfh)\ll X^{s-1+\eps}$.
\end{theorem}

The upper bound presented in this theorem saves a factor $X^{1-\eps}$ beyond 
square-root cancellation, improving on the factor $X^{1/2-\eps}$ visible in (\ref{1.3}). Such 
a conclusion lies beyond any anticipated by Brandes and Hughes (see the discussion 
concluding \cite{BH2021}). Moreover, as we show in Theorem \ref{theorem7.1}, there exist 
$k$-tuples $\bfh\ne {\mathbf 0}$ having the property that $J_{s,k}(X;\bfh)\gg X^{s-1}$, 
so the conclusion of Theorem \ref{theorem1.2} is in some respects best possible.\par

Our next theorem shows that $J_{s,k}(X;\bfh)=o\left(X^s\right)$ whenever 
$s<\tfrac{1}{2}k(k+1)$ and $(h_1,\ldots ,h_{k-1})\ne {\mathbf 0}$, improving on an 
earlier result of \cite{BH2021}.

\begin{theorem}\label{theorem1.3} Suppose that $k\ge 3$ and 
$\bfh\in \dbZ^k\setminus \{{\mathbf 0}\}$. Let $l$ be the smallest index having the 
property that $h_l\ne 0$. Then whenever $l<k$ and $s$ is a natural number satisfying 
$s<\tfrac{1}{2}k(k+1)$, one has
\[
J_{s,k}(X;\bfh)\ll X^\eps\left(X^{s-1/2}+X^{s-\del(s,k,l)}\right),
\]
where
\[
\del(s,k,l)=\tfrac{1}{2}(k-l)(k-l+1)\biggl( \frac{k(k+1)-2s}{k(k+1)-l(l+1)}\biggr) .
\]
\end{theorem}

A conclusion analogous to that of Theorem \ref{theorem1.3} is obtained in 
\cite[Theorem 1]{BH2021}, though with the weaker exponent
\[
\del(s,k,l)=\tfrac{1}{2}(k-l)(k-l+1)\biggl( \frac{k(k+1)-2s}{k(k+1)}\biggr) .
\]

\par The conclusions of Theorems \ref{theorem1.1}, \ref{theorem1.2} and 
\ref{theorem1.3} have nothing to say concerning $J_{s,k}(X;\bfh)$ at the critical exponent 
$s=\tfrac{1}{2}k(k+1)$. Readers less familiar with the nuances of Vinogradov's mean value 
theorem may care to note in this context that when $s>\tfrac{1}{2}k(k+1)$, then an 
application of the circle method delivers an asymptotic formula of the shape 
$J_{s,k}(X;\bfh)\sim C(\bfh)X^{2s-k(k+1)/2}$, where $C(\bfh)$ is positive provided that 
$\bfh$ satisfies appropriate local solubility conditions. Since the main term here is larger 
than the square-root of the number of available choices for the underlying variables, this 
situation with $s>\tfrac{1}{2}k(k+1)$ does not require subconvexity in its treatment. In 
contrast, when $s=\tfrac{1}{2}k(k+1)$, one requires subconvex minor arc estimates in 
order to show that the expected product of local densities delivers the anticipated 
asymptotic formula.\par

In recent work concerning the cubic case of the inhomogeneous Vinogradov system, the 
author applied the Hardy-Littlewood method to obtain an asymptotic formula for 
$J_{6,3}(X;\bfh)$ when $h_1\ne 0$ (see \cite[Theorem 1.1]{Woo2021c}). Moreover, when 
$h_1=0$ and $h_2\ne 0$, an asymptotic formula for $J_{6,3}(X;\bfh)$ is obtained in 
\cite[Theorem 1.2]{Woo2021c} provided that $X$ is sufficiently large in terms of $h_2$. 
Both conclusions depend on minor arc estimates with better than square-root cancellation. 
When the degree $k$ exceeds $3$, such conclusions are beyond the reach of current 
technology. Nonetheless, by application of conjectural mean value estimates potentially 
within reach of efficient congruencing and decoupling methods, some progress is possible.
\par

In order to describe the asymptotic formula associated with $J_{s,k}(X;\bfh)$ at the critical 
point $s=k(k+1)/2$, we introduce some notation. We write $B_k(X;\bfh)$ for 
$J_{k(k+1)/2,k}(X;\bfh)$. Next, we introduce the generating functions
\begin{equation}\label{1.4}
I(\bfbet)=\int_0^1e(\bet_1\gam+\ldots +\bet_k\gam^k)\d\gam
\end{equation}
and
\begin{equation}\label{1.5}
S(q,\bfa)=\sum_{r=1}^q e_q(a_1r+\ldots +a_kr^k),
\end{equation}
in which we write $e(z)$ for $e^{2\pi iz}$ and use $e_q(u)$ as shorthand for 
$e^{2\pi i u/q}$. Putting $n_j=h_jX^{-j}$ $(1\le j\le k)$, we define the {\it singular 
integral}
\begin{equation}\label{1.6}
\grJ_k(\bfh)=\int_{\dbR^k}|I(\bfbet )|^{k(k+1)}e(-\bfbet \cdot \bfn)\d\bfbet ,
\end{equation}
in which $\bfbet \cdot \bfn$ denotes $\bet_1n_1+\ldots +\bet_kn_k$. Finally, we define 
the {\it singular series}
\begin{equation}\label{1.7}
\grS_k(\bfh)=\sum_{q=1}^\infty 
\sum_{\substack{1\le \bfa\le q\\ (q,a_1,\ldots a_k)=1}}
\left|q^{-1}S(q,\bfa)\right|^{k(k+1)}e_q(-\bfa \cdot \bfh).
\end{equation}
We note that both the singular integral $\grJ_k(\bfh)$ and the singular series 
$\grS_k(\bfh)$ are known to converge absolutely (see \cite[Theorem 1]{Ark1984} or 
\cite[Theorem 3.7]{AKC2004}).\par

Our progress is conditional on the extended main conjecture in Vinogradov's mean value 
theorem (Conjecture \ref{conjecture8.1}). Once again, the conclusion of the next theorem 
implicitly encodes a minor arc estimate beyond the convexity barrier.

\begin{theorem}\label{theorem1.4} Assume the extended main conjecture in Vinogradov's 
mean value theorem. Suppose that $\bfh\in \dbZ^k$ and $h_l\ne 0$ for some index $l$ 
with $1\le l<k$. Then provided that $X$ is sufficiently large in terms of $\bfh$, one has
\[
B_k(X;\bfh)=\grJ_k(\bfh)\grS_k(\bfh)X^{k(k+1)/2}+o(X^{k(k+1)/2}),
\]
in which $0\le \grJ_k(\bfh)\ll 1$ and $0\le \grS_k(\bfh)\ll 1$.
\end{theorem}

We turn now to the topic of paucity and its relation to inhomogeneous Vinogradov systems. 
When the number of variables in the Vinogradov system (\ref{1.1}) is small, one may 
obtain estimates for $J_{s,k}(X;\bfh)$ far below the convexity barrier. That such should be 
possible is apparent from recent work of the author \cite{Woo2021a} concerning paucity in 
relatives of Vinogradov's mean value theorem. Consider, by way of an illustrative example, 
the system of equations
\begin{equation}\label{1.8}
x_1^j+\ldots +x_k^j=y_1^j+\ldots +y_k^j\quad (\text{$1\le j\le k$, $j\ne k-d$}),
\end{equation}
in which $k\ge 4$ and the non-negative integer $d$ is fixed. Let $I^*_{k,d}(X)$ denote the 
number of integral solutions of (\ref{1.8}) with $1\le \bfx,\bfy\le X$ in which 
$(x_1,\ldots ,x_k)$ is not a permutation of $(y_1,\ldots ,y_k)$. Then 
\cite[Corollary 1.2]{Woo2021a} shows that when $d=o(k^{1/4})$, one has 
$I^*_{k,d}(X)\ll X^{(2+o(1))\sqrt{k}}$. However, should one have
\[
\sum_{i=1}^k(x_i^{k-d}-y_i^{k-d})=h_{k-d}\ne 0,
\]
then $(x_1,\ldots ,x_k)$ cannot be a permutation of $(y_1,\ldots ,y_k)$. Thus we conclude 
that when $\bfh=(0,\ldots ,0,h_{k-d},0,\ldots ,0)$, with $h_{k-d}\ne 0$, then
\[
J_{k,k}(X;\bfh)\ll X^{(2+o(1))\sqrt{k}}.
\]
By elaborating on these ideas, non-trivial estimates may be obtained without restriction on 
$d$.

\begin{theorem}\label{theorem1.5} Suppose that $k\ge 3$ and 
$\bfh\in \dbZ^k\setminus \{{\mathbf 0}\}$. Suppose further that for some index $l$ with 
$2\le l\le k$ one has $h_l\ne 0$, but that $h_j=0$ when $j\ne l$ and $1\le j\le k$. Then one 
has $J_{k,k}(X;\bfh)\ll X^{k-l+1+\eps}$.
\end{theorem}

The conclusion of this theorem yields stronger bounds than any supplied by Theorems 
\ref{theorem1.1} and \ref{theorem1.2} when $l\ge 3$.\par

This paper is organised as follows. In \S2 we adapt the author's work on the asymptotic 
formula in Waring's problem \cite{Woo2012b} to bound Fourier coefficients associated with 
the inhomogeneous Vinogradov system (\ref{1.1}). This approach has a significant 
advantage over the corresponding analysis of \cite{BH2021}, which is that the mean values 
of interest may be restricted to subsets of $[0,1)^k$, such as sets of minor arcs of use in 
applications of the Hardy-Littlewood method. We apply this method in combination with 
H\"older's inequality, relating $J_{s,k}(X;\bfh)$ to mixed mean value estimates more 
efficient than the simple ones considered in \cite{BH2021}. These mixed mean values are 
examined in \S3, preparing the ground in \S4 for the proof of our simplest subconvex 
bounds described in Theorems \ref{theorem1.1}, \ref{theorem1.2} and \ref{theorem1.3}. 
Preparations for the proof of Theorem \ref{theorem1.4} are presented in \S5, where we 
apply the extended main conjecture in Vinogradov's mean value theorem as the key input to 
provide subconvex minor arc estimates. The application of the Hardy-Littlewood method 
itself is described in \S6, where the proof of Theorem \ref{theorem1.4} is completed. In \S7 
we explore the application of ideas from the theory of paucity to bounds for 
$J_{s,k}(X;\bfh)$, and in particular we prove Theorem \ref{theorem1.5}. Finally, in the 
appendix attached as \S8, we discuss the extended main conjecture in Vinogradov's mean 
value theorem and its immediate applications to generalisations of small cap estimates.\par

Our basic parameter is $X$, a sufficiently large positive number. Whenever $\eps$ appears 
in a statement, either implicitly or explicitly, we assert that the statement holds for each 
$\eps>0$. In this paper, implicit constants in Vinogradov's notation $\ll$ and $\gg$ may 
depend on $\eps$, $k$ and $s$. We make use of vector notation in the form 
$\bfx=(x_1,\ldots,x_r)$, the dimension $r$ depending on the course of the argument. We 
also write $(a_1,\ldots ,a_s)$ for the greatest common divisor of the integers 
$a_1,\ldots ,a_s$, any ambiguity between ordered $s$-tuples and corresponding greatest 
common divisors being easily resolved by context. Finally, we write $\|\tet\|$ for 
$\min\{|\tet-m|:m\in \dbZ\}$.

\section{Auxiliary mean values utilising shifts} We first establish a reasonably flexible mean 
value estimate by applying ideas underlying our recent work on the Hilbert-Kamke problem, 
as modified to handle the cubic case of the inhomogeneous Vinogradov system (see 
\cite[Theorem 2.1]{Woo2021b} and \cite[Lemma 2.1]{Woo2021c}). This argument has its 
genesis in earlier work of the author concerning the asymptotic formula in Waring's problem 
(see \cite[Lemma 10.1]{Woo2012a} and \cite[Theorem 2.1]{Woo2012b}). Define 
$f(\bfalp;X)=f_k(\bfalp;X)$ by
\begin{equation}\label{2.1}
f_k(\bfalp;X)=\sum_{1\le x\le X}e(\alp_1x+\ldots +\alp_kx^k).
\end{equation}
Then, when $\bfh\in \dbZ^k$ and $\grB\subseteq \dbR$ is measurable, we put
\begin{equation}\label{2.2}
I_s(\grB;X;\bfh)=\int_\grB\int_{[0,1)^{k-1}}|f_k(\bfalp;X)|^{2s}e(-\bfalp\cdot \bfh)
\d\bfalp ,
\end{equation}
in which $\bfalp\cdot \bfh=\alp_1h_1+\ldots +\alp_kh_k$ and $\d\bfalp$ denotes 
$\d\alp_1\cdots \d\alp_k$. Note that by orthogonality, one then has
\begin{equation}\label{2.3}
J_{s,k}(X;\bfh)=I_s([0,1);X;\bfh).
\end{equation}
We also make use of the generating function $g(\bfalp,\tet;X)=g_k(\bfalp,\tet;X)$ defined 
by putting
\begin{equation}\label{2.4}
g_k(\bfalp, \tet;X)=\sum_{1\le y\le X}e\left( y\tet+\nu_2(y;\bfh)\alp_2+\ldots 
+\nu_k(y;\bfh)\alp_k\right) ,
\end{equation}
in which
\begin{equation}\label{2.5}
\nu_j(y;\bfh)=\sum_{i=0}^{j-1}\binom{j}{i}h_{j-i}y^i\quad (1\le j\le k).
\end{equation}

\begin{lemma}\label{lemma2.1} Suppose that $s\in \dbN$, $\bfh\in \dbZ^k$ and 
$\grB\subseteq \dbR$ is measurable. Then
\[
I_s(\grB;X;\bfh)\ll X^{-1}(\log X)^{2s}\sup_{\Gam\in [0,1)}\int_\grB \int_{[0,1)^{k-1}}
|f_k(\bfalp;2X)^{2s}g_k(\bfalp,\Gam;X)|\d\bfalp .
\]
\end{lemma}

\begin{proof} The argument we present here is very similar to that underlying the proof of 
\cite[Lemma 2.1]{Woo2021c}, though there are sufficiently many differences that a full 
account seems warranted. We first reformulate the mean value $I_s(\grB;X;\bfh)$ defined 
in (\ref{2.2}) in preparation for the exploitation of a shift in the underlying variables. Write 
$\psi(u;\bftet)=\tet_1u+\ldots +\tet_ku^k$. Then, as in the analogous argument of 
\cite[Lemma 2.1]{Woo2021c}, it follows via orthogonality that for every integral shift $y$ 
with $1\le y\le X$, one has
\[
f_k(\bfalp;X)=\int_0^1\grf_y(\bfalp;\gam)K(\gam)\d\gam ,
\]
where
\begin{equation}\label{2.6}
\grf_y(\bfalp;\gam)=\sum_{1\le x\le 2X}e\left( \psi(x-y;\bfalp)+\gam(x-y)\right)
\end{equation}
and
\[
K(\gam)=\sum_{1\le z\le X}e(-\gam z).
\]
Write
\begin{equation}\label{2.7}
\grF_y(\bfalp;\bfgam)=\prod_{i=1}^s \grf_y(\bfalp;\gam_i)\grf_y(-\bfalp;-\gam_{s+i}),
\end{equation}
\[
\Ktil(\bfgam)=\prod_{i=1}^sK(\gam_i)K(-\gam_{s+i}),
\]
and
\begin{equation}\label{2.8}
\calI(\bfgam;y;\bfh)=\int_\grB \int_{[0,1)^{k-1}}\grF_y(\bfalp;\bfgam)e(-\bfalp \cdot \bfh)
\d\bfalp .
\end{equation}
Then we infer from (\ref{2.2}) that
\begin{equation}\label{2.9}
I_s(\grB;X;\bfh)=\int_{[0,1)^{2s}}\calI(\bfgam ;y;\bfh)\Ktil(\bfgam)\d\bfgam .
\end{equation}

\par Write $\d\bfalp_{k-1}$ as shorthand for $\d\alp_1\cdots \d\alp_{k-1}$. Then, by 
orthogonality, we discern from (\ref{2.7}) that
\begin{equation}\label{2.10}
\int_{[0,1)^{k-1}}\grF_y(\bfalp;\bfgam)e(-\bfalp \cdot \bfh)\d\bfalp_{k-1}
=\sum_{1\le \bfx\le 2X}\Del(\alp_k,\bfgam;\bfh,y),
\end{equation}
where $\Del(\alp_k,\bfgam;\bfh,y)$ is equal to
\[
e\biggl( \sum_{i=1}^s \left( \alp_k\left( (x_i-y)^k-(x_{s+i}-y)^k\right)+\left( \gam_i(x_i-y)
-\gam_{s+i}(x_{s+i}-y)\right)\right)-\alp_kh_k\biggr) ,
\]
when
\begin{equation}\label{2.11}
\sum_{i=1}^s\left( (x_i-y)^j-(x_{s+i}-y)^j\right)=h_j\quad (1\le j\le k-1),
\end{equation}
and otherwise $\Del(\alp_k,\bfgam;\bfh,y)$ is equal to $0$.\par

By applying the binomial theorem within (\ref{2.11}), we obtain the relations
\[
\sum_{i=1}^s(x_i^j-x_{s+i}^j)=\sum_{l=0}^{j-1}\binom{j}{l}h_{j-l}y^l\quad 
(1\le j\le k-1),
\]
and
\[
\sum_{i=1}^s(x_i^k-x_{s+i}^k)=\sum_{l=1}^{k-1}\binom{k}{l}h_{k-l}y^l
+\sum_{i=1}^s\left( (x_i-y)^k-(x_{s+i}-y)^k\right) .
\]
Define
\begin{equation}\label{2.12}
G(\bfalp;\bfh;\bfgam)=\sum_{1\le y\le X}
e\biggl( -\sum_{j=1}^k\alp_j\sum_{l=0}^{j-1}\binom{j}{l}
h_{j-l}y^l-y\Gam(\bfgam)\biggr),
\end{equation}
where
\[
\Gam(\bfgam)=\sum_{i=1}^s(\gam_i-\gam_{s+i}).
\]
Then we find from (\ref{2.7}) and (\ref{2.10}) that
\[
\sum_{1\le y\le X}\int_{[0,1)^{k-1}}\grF_y(\bfalp;\bfgam)
e(-\bfalp \cdot \bfh)\d\bfalp_{k-1}=\int_{[0,1)^{k-1}}\grF_0(\bfalp;\bfgam)
G(\bfalp;\bfh;\bfgam)\d\bfalp_{k-1}.
\]
On substituting this last relation into (\ref{2.8}) and thence into (\ref{2.9}), we obtain
\begin{align}
I_s(\grB;X;\bfh)&=\lfloor X\rfloor^{-1}\sum_{1\le y\le X}\int_{[0,1)^{2s}}
\calI(\bfgam;y;\bfh)\Ktil(\bfgam)\d\bfgam \notag\\
&\ll X^{-1}\int_{[0,1)^{2s}}|H(\bfgam)\Ktil(\bfgam)|\d\bfgam ,\label{2.13}
\end{align}
where
\begin{equation}\label{2.14}
H(\bfgam)=\int_\grB \int_{[0,1)^{k-1}}\grF_0(\bfalp;\bfgam)G(\bfalp;\bfh;\bfgam)\d\bfalp 
.
\end{equation}

\par Note next from (\ref{2.1}) and (\ref{2.6}) that $\grf_0(\bfalp;0)=f_k(\bfalp;2X)$. 
Thus, from the elementary inequality $|z_1\cdots z_n|\le |z_1|^n+\ldots +|z_n|^n$ and 
(\ref{2.7}), we deduce that
\[
|\grF_0(\bfalp;\bfgam)|\le \sum_{i=1}^{2s}|\grf_0(\bfalp;\gam_i)|^{2s}=
\sum_{i=1}^{2s}|f_k(\alp_k,\ldots ,\alp_2,\alp_1+\gam_i;2X)|^{2s}.
\]
Moreover, in view of (\ref{2.5}), we have $\nu_1(y;\bfh)=h_1$, so that 
$\nu_1(y;\bfh)$ is independent of $y$. Then, from (\ref{2.4}), (\ref{2.5}) and (\ref{2.12}) 
we see that
\[
|G(\bfalp;\bfh;\bfgam)=|g_k(\bfalp,\Gam(\bfgam);X)|.
\]
Define
\[
U_s(\grB)=\sup_{\Gam \in [0,1)}\int_\grB \int_{[0,1)^{k-1}}|f_k(\bfalp;2X)^{2s}
g_k(\bfalp,\Gam;X)|\d\bfalp ,
\]
and observe that $|G(\bfalp;\bfh;\bfgam)|$ is independent of $\alp_1$. Then a change of 
variable leads from (\ref{2.14}) to the upper bound
\begin{equation}\label{2.15}
|H(\bfgam)|\le \int_\grB \int_{[0,1)^{k-1}}|f_k(\bfalp;2X)^{2s}
g_k(\bfalp,\Gam(\bfgam);X)|\d\bfalp \le U_s(\grB).
\end{equation}
Recall next that
\[
\int_0^1|K(\gam)|\d\gam \ll \int_0^1\min\{ X,\|\gam\|^{-1}\}\d\gam \ll \log (2X).
\]
Then we perceive from (\ref{2.13}) and (\ref{2.15}) that
\[
I_s(\grB;X;\bfh)\ll X^{-1}U_s(\grB)\biggl( \int_0^1|K(\gam)|\d\gam \biggr)^{2s}\ll 
X^{-1}(\log (2X))^{2s}U_s(\grB).
\]
This completes the proof of the lemma.
\end{proof}

\section{Mixed mean value estimates} In this section we derive mixed mean value 
estimates involving $f(\bfalp;2X)$ and $g(\bfalp, \Gam;X)$. We apply these estimates 
in \S 4 to establish Theorems \ref{theorem1.1}, \ref{theorem1.2} and \ref{theorem1.3}. 
Throughout, we abbreviate $g_k(\bfalp,0;X)$ to $g_k(\bfalp;X)$.

\begin{lemma}\label{lemma3.1} Suppose that $k\ge 3$ and 
$\bfh\in \dbZ^k\setminus \{{\mathbf 0}\}$. Let $l$ be the smallest index having the 
property that $h_l\ne 0$. Then whenever $l<k$, and $u$ and $r$ are non-negative integers 
with $u\le l(l+1)/2$, one has
\[
\int_{[0,1)^k}|g_k(\bfalp;X)^{2r}f_k(\bfalp;2X)^{2u}|\d\bfalp \ll X^{u+\eps}\left( 
X^r+X^{2r-(k-l)(k-l+1)/2}\right).
\]
\end{lemma}

\begin{proof} Suppose that $u$ is an integer with $0\le u\le l(l+1)/2$, and define
\[
I=\int_{[0,1)^k}|g_k(\bfalp;X)^{2r}f_k(\bfalp;2X)^{2u}|\d\bfalp .
\]
Recall the definition (\ref{2.5}) of the polynomial $\nu_j(y;\bfh)$ $(1\le j\le k)$. Then, by 
orthogonality, the mean value $I$ counts the integral solutions of the system
\begin{equation}\label{3.1}
\sum_{i=1}^u(x_i^j-y_i^j)=\sum_{m=1}^r\left( \nu_j(w_m;\bfh)-\nu_j(z_m;\bfh)\right) 
\quad (1\le j\le k),
\end{equation}
with $1\le \bfx,\bfy\le 2X$ and $1\le \bfw,\bfz\le X$. Our hypothesis that $h_j=0$ for 
$1\le j<l$ ensures that
\[
\nu_j(w_m;\bfh)-\nu_j(z_m;\bfh)=0\quad (1\le j\le l),
\]
and so we deduce from (\ref{3.1}) that
\begin{equation}\label{3.2}
\sum_{i=1}^u(x_i^j-y_i^j)=0\quad (1\le j\le l).
\end{equation}
Since $u\le l(l+1)/2$, we therefore deduce from the (now proven) main conjecture in 
Vinogradov's mean value theorem (see \cite{BDG2016} and \cite{Woo2016, Woo2019}) 
that the number of choices $I_0$ for $\bfx$ and $\bfy$ satisfies
\begin{equation}\label{3.3}
I_0\le J_{u,l}(2X;{\mathbf 0})\ll X^{u+\eps}.
\end{equation}

\par Fix any choice of $\bfx$ and $\bfy$ satisfying (\ref{3.2}). Then by taking appropriate 
linear combinations of the equations (\ref{3.1}), we find that there are integers 
$n_j=n_j(\bfx,\bfy;\bfh)$ for which
\begin{equation}\label{3.4}
\sum_{m=1}^r(w_m^j-z_m^j)=n_j(\bfx,\bfy;\bfh)\quad (1\le j\le k-l).
\end{equation}
For each fixed choice of $\bfx$ and $\bfy$, denote by $I_1=I_1(\bfx,\bfy)$ the number of 
solutions of the system (\ref{3.4}) with $1\le \bfz,\bfw\le X$. Then by orthogonality and the 
triangle inequality, we find that
\begin{align*}
I_1(\bfx,\bfy)&=\int_{[0,1)^{k-l}}|f_{k-l}(\bfalp;X)|^{2r}e(-\bfalp \cdot \bfn)\d\bfalp \\
&\le \int_{[0,1)^{k-l}}|f_{k-l}(\bfalp;X)|^{2r}\d\bfalp .
\end{align*}
Thus, again applying the (now proven) main conjecture in Vinogradov's mean value 
theorem, we infer that
\[
I_1(\bfx,\bfy)\le J_{r,k-l}(X;{\mathbf 0})\ll X^\eps\left( X^r+X^{2r-(k-l)(k-l+1)/2}\right) .
\]
On recalling (\ref{3.3}), we therefore conclude that
\[
I\le I_0\max_{\bfx,\bfy}I_1(\bfx,\bfy)\ll X^{u+\eps}\left( X^r+X^{2r-(k-l)(k-l+1)/2}
\right) .
\]
This completes the proof of the lemma.
\end{proof}

As an immediate consequence of the upper bound of Lemma \ref{lemma3.1}, we record 
the following estimate in which $g_k(\bfalp,\Gam;X)$ substitutes for $g_k(\bfalp;X)$.

\begin{lemma}\label{lemma3.2} Suppose that $k\ge 3$ and $\bfh\in \dbZ^k\setminus 
\{{\mathbf 0}\}$. Let $l$ be the smallest index having the property that $h_l\ne 0$. Then 
whenever $l<k$, and $u$ and $r$ are non-negative integers with $u\le l(l+1)/2$, one has
\[
\sup_{\Gam\in [0,1)}\int_{[0,1)^k}|g_k(\bfalp,\Gam;X)^{2r}f_k(\bfalp;2X)^{2u}|\d\bfalp 
\ll X^{u+\eps}\left( X^r+X^{2r-(k-l)(k-l+1)/2}\right).
\]
\end{lemma}

\begin{proof} The mean value
\[
\Tet_{u,r}(\Gam;X)=\int_{[0,1)^k}|g_k(\bfalp,\Gam;X)^{2r}f_k(\bfalp;2X)^{2u}|\d\bfalp
\]
counts the integral solutions $\bfx,\bfy,\bfz,\bfw$ of the system (\ref{3.1}) with weight
\[
e\biggl( -\Gam \sum_{m=1}^r(w_m-z_m)\biggr).
\]
Since this weight is unimodular, we deduce via orthogonality that
\[
\sup_{\Gam\in [0,1)}\Tet_{u,r}(\Gam;X)\le \int_{[0,1)^k}|g_k(\bfalp;X)^{2r}
f_k(\bfalp;2X)^{2u}|\d\bfalp .
\]
The conclusion of the lemma is now immediate from that of Lemma \ref{lemma3.1}.
\end{proof}

We next prepare for the proof of Theorem \ref{theorem1.2}. When $s\in \dbN$, write
\begin{equation}\label{3.5}
\Psi_s(\Gam;X)=\int_{[0,1)^k}|f_k(\bfalp;2X)^{2s}g_k(\bfalp,\Gam;X)|\d\bfalp .
\end{equation}

\begin{lemma}\label{lemma3.3} Suppose that $k\ge 3$, and $\bfh\in \dbZ^k\setminus 
\{ {\mathbf 0}\}$ satisfies the condition that $h_j=0$ for $j\le k-2$, but $h_{k-1}\ne 0$. 
Then whenever $s$ is an integer with $1\le s\le k(k-1)/2$, one has
\[
\sup_{\Gam\in [0,1)}\Psi_s(\Gam;X)\ll X^{s+\eps}.
\]
\end{lemma}

\begin{proof} With the hypotheses on $\bfh$ available from the statement of the lemma, it 
follows from (\ref{2.4}) and (\ref{2.5}) that
\begin{align*}
|g_k(\bfalp,\Gam;X)|&=\biggl| \sum_{1\le y\le X}e\left( (kh_{k-1}\alp_k+\Gam)y\right) 
\biggr| \\
&\ll \min\{X,\|kh_{k-1}\alp_k+\Gam\|^{-1}\}.
\end{align*}
Since $|g_k(\bfalp,\Gam;X)|$ is independent of $\alp_1,\ldots ,\alp_{k-1}$, we therefore 
perceive via orthogonality that
\begin{equation}\label{3.6}
\int_{[0,1)^{k-1}}|f_k(\bfalp;2X)^{2s}g_k(\bfalp,\Gam;X)|\d\bfalp_{k-1}\ll 
T\min \{ X,\|kh_{k-1}\alp_k+\Gam\|^{-1}\},
\end{equation}
where
\[
T=\int_{[0,1)^{k-1}}|f_k(\bfalp;2X)|^{2s}\d\bfalp_{k-1}
\]
counts the number of integral solutions of the system of equations
\[
\sum_{i=1}^s(x_i^j-y_i^j)=0\quad (1\le j\le k-1),
\]
with $1\le \bfx,\bfy\le X$, each solution being counted with weight
\begin{equation}\label{3.7}
e\biggl( \alp_k\sum_{i=1}^s(x_i^k-y_i^k)\biggr) .
\end{equation}

\par Since the weight (\ref{3.7}) is unimodular, we see that $T\ll J_{s,k-1}(X;\bf0)$. From 
the (now proven) main conjecture in Vinogradov's mean value theorem, we thus deduce 
from (\ref{3.5}) and (\ref{3.6}) via a change of variables that 
\begin{align*}
\Psi_s(\Gam;X)&\ll X^{s+\eps}\int_0^1\min \{ X,\|kh_{k-1}\alp_k+\Gam\|^{-1}\}\d\alp_k 
\\
&=X^{s+\eps}\int_0^1\min\{ X,\|\bet\|^{-1}\}\d\bet .
\end{align*}
Thus we conclude that
\[
\sup_{\Gam\in [0,1)}\Psi_s(\Gam;X)\ll X^{s+\eps}\log (2X)\ll X^{s+2\eps},
\]
and the proof of the lemma is complete.
\end{proof}

\section{The simplest subconvex bounds} We now attend to the matter of converting the 
auxiliary estimates of \S3, using the apparatus prepared in \S2, so as to establish 
Theorems \ref{theorem1.1}, \ref{theorem1.2} and \ref{theorem1.3}. We begin with the 
proof of Theorems \ref{theorem1.1} and \ref{theorem1.3}. We should emphasise here that 
our formulation of Lemma \ref{lemma2.1}, which we will shortly wield in earnest, bounds 
$I_s(\grB;X;\bfh)$ for any measurable set $\grB$. In the proofs of Theorems 
\ref{theorem1.1}, \ref{theorem1.2} and \ref{theorem1.3}, we make use of this lemma 
only when $\grB=[0,1)$. In such circumstances, one could do away with the elaborate 
arguments employed in the proof of Lemma \ref{lemma2.1}, making do only with simple 
arguments counting solutions of Diophantine systems. Indeed, this was the approach taken 
in the proof of \cite[Lemma 10.1]{Woo2012a}. We will, however, need the full force of 
Lemma \ref{lemma2.1} in \S5, and we express the hope that the greater flexibility of this 
lemma may inspire future refinement even to the results established in the present section. 

\begin{proof}[The proof of Theorems \ref{theorem1.1} and \ref{theorem1.3}] We suppose 
that $k\ge 3$, $\bfh\in \dbZ^k\setminus \{ {\mathbf 0}\}$, and that $l$ is the smallest 
index with $1\le l\le k$ having the property that $h_l\ne 0$. The hypotheses of Theorems 
\ref{theorem1.1} and \ref{theorem1.3} then permit us the assumption that $l<k$. On 
recalling (\ref{2.3}), we see that Lemma \ref{lemma2.1} delivers the bound
\begin{equation}\label{4.1}
J_{s,k}(X;\bfh)\ll X^{-1}(\log X)^{2s}\sup_{\Gam\in [0,1)}V_s(\Gam),
\end{equation}
where
\begin{equation}\label{4.2}
V_s(\Gam)=\int_{[0,1)^k}|f_k(\bfalp;2X)^{2s}g_k(\bfalp,\Gam;X)|\d\bfalp .
\end{equation}

\par We pursue two different analyses of the mean value (\ref{4.2}), the first of which 
delivers Theorem \ref{theorem1.1}, and the second Theorem \ref{theorem1.3}. Write 
$R=(k-l)(k-l+1)$,
\begin{equation}\label{4.3}
u=\min \left\{ sR,l(l+1)/2\right\}\quad \text{and}\quad v=\frac{s-u/R}{1-1/R}.
\end{equation}
Then it follows from an application of H\"older's inequality in (\ref{4.2}) that
\begin{equation}\label{4.4}
V_s(\Gam)\le W_{u,1}(\Gam)^{1/R}W_{v,2}^{1-1/R},
\end{equation}
where
\begin{equation}\label{4.5}
W_{u,1}(\Gam)=\int_{[0,1)^k}|g_k(\bfalp,\Gam;X)^Rf_k(\bfalp;2X)^{2u}|\d\bfalp 
\end{equation}
and
\begin{equation}\label{4.6}
W_{v,2}=\int_{[0,1)^k}|f_k(\bfalp;2X)|^{2v}\d\bfalp .
\end{equation}

\par Since $R=(k-l)(k-l+1)$ is an even integer and $1\le u\le l(l+1)/2$, an application of 
Lemma \ref{lemma3.2} to (\ref{4.5}) reveals that
\begin{equation}\label{4.7}
\sup_{\Gam\in [0,1)}W_{u,1}(\Gam)\ll X^{u+\frac{1}{2}R+\eps}.
\end{equation}
Moreover, by applying the (now proven) main conjecture in Vinogradov's mean value 
theorem, it follows from (\ref{4.6}) that
\begin{equation}\label{4.8}
W_{v,2}\ll X^\eps (X^v+X^{2v-k(k+1)/2}).
\end{equation}
We therefore deduce from (\ref{4.3}) and (\ref{4.4}) that
\begin{align*}
V_s(\Gam)&\ll X^\eps \left( X^{u+R/2}\right)^{1/R}\left( X^v+X^{2v-k(k+1)/2}
\right)^{1-1/R}\\
&\ll X^{s+\frac{1}{2}+\tet+\eps},
\end{align*}
where
\[
\tet=(1-1/R)\max\{ 0,v-k(k+1)/2\} .
\]
By substituting this estimate for $V_s(\Gam)$ into (\ref{4.1}), we conclude thus far that
\begin{equation}\label{4.9}
J_{s,k}(X;\bfh)\ll X^{s-\frac{1}{2}+\tet+\eps}.
\end{equation}
One has $\tet=0$ when $v\le k(k+1)/2$, and by (\ref{4.3}) such is the case so long as
\[
s\le \frac{u}{R}+\tfrac{1}{2}k(k+1)(1-1/R),
\]
a constraint guaranteed to hold provided that the hypothesis (\ref{1.2}) is in force. The 
conclusion (\ref{1.3}) of Theorem \ref{theorem1.1} therefore follows at once from 
(\ref{4.9}) in this situation in which $\tet=0$. Notice here that
\[
\tfrac{1}{2}k(k+1)-\frac{k(k+1)-l(l+1)}{2(k-l)(k-l+1)}\le \tfrac{1}{2}k(k+1)-1
\]
if and only if $2(k-l)(k-l+1)\le k(k+1)-l(l+1)$, confirming that the condition (\ref{1.2}) is met 
for all $s<k(k+1)/2$ whenever $(k+1-3l)(k-l)\le 0$. This confirms that the estimate 
(\ref{1.3}) does indeed hold, when $s<k(k+1)/2$, provided that $1\le l\le (k+1)/3$, 
completing the proof of Theorem \ref{theorem1.1}.\par

We now turn to the proof of Theorem \ref{theorem1.3}. Here, in view of the conclusion of 
Theorem \ref{theorem1.1}, we already have the bound 
$J_{s,k}(X;\bfh)\ll X^{s-1/2+\eps}$ when the constraint (\ref{1.2}) is in force. We may 
therefore suppose henceforth that
\[
\tfrac{1}{2}k(k+1)-\frac{k(k+1)-l(l+1)}{2R}<s<\tfrac{1}{2}k(k+1),
\]
whence
\begin{equation}\label{4.10}
\frac{k(k+1)-2s}{k(k+1)-l(l+1)}<\frac{1}{R}.
\end{equation}
Write
\[
u=\tfrac{1}{2}l(l+1),\quad v=\tfrac{1}{2}k(k+1)\quad \text{and}\quad 
a=\frac{k(k+1)-2s}{2(v-u)},
\]
and recall the notation introduced in (\ref{4.5}) and (\ref{4.6}). Then since it follows from 
(\ref{4.10}) that $0\le a<1/R$, an application of H\"older's inequality in (\ref{4.2}) shows 
in this situation that
\[
V_s(\Gam)\le \biggl( \sup_{\bfalp\in [0,1)^k}|g_k(\bfalp,\Gam;X)|\biggr)^{1-Ra}
W_{u,1}(\Gam)^aW_{v,2}^{1-a}.
\]
Thus, we deduce from (\ref{4.7}) and (\ref{4.8}) that
\begin{align*}
\sup_{\Gam\in [0,1)}V_s(\Gam)&\ll X^{1-Ra+\eps}\left( X^{u+R/2}\right)^a
(X^v)^{1-a}\\
&\ll X^{s+1-\frac{1}{2}Ra+\eps}.
\end{align*}

\par We therefore conclude from (\ref{4.1}) that in this scenario, one has
\[
J_{s,k}(X;\bfh)\ll X^{s-\del(s,k,l)+\eps},
\]
where
\[
\del(s,k,l)=\tfrac{1}{2}Ra=\tfrac{1}{2}(k-l)(k-l+1)\frac{k(k+1)-2s}{k(k+1)-l(l+1)}.
\]
This completes the proof of Theorem \ref{theorem1.3}.
\end{proof}

We complete this section by establishing Theorem \ref{theorem1.2}, exploiting the fact that 
when $h_j=0$ for $1\le j\le k-2$ and $h_{k-1}\ne 0$, then the generating function 
$g_k(\bfalp,\tet;X)$ is a linear exponential sum in the underlying variable $y$.

\begin{proof}[The proof of Theorem \ref{theorem1.2}] We begin with a preliminary 
simplification. We work under the hypotheses of the statement of Theorem 
\ref{theorem1.2}, and consider the $(l+1)$-tuple $\bfh'=(h_1,\ldots ,h_{l+1})$, in which 
we may assume that $h_j=0$ for $1\le j<l$ but $h_l\ne 0$. By discarding the equations in 
(\ref{1.1}) of degree exceeding $l+1$, we see that $J_{s,k}(X;\bfh)\le J_{s,l+1}(X;\bfh')$. 
But when $1\le s\le l(l+1)/2$, it follows from (\ref{2.3}) by applying Lemma 
\ref{lemma2.1} with $l+1$ in place of $k$ that
\[
J_{s,l+1}(X;\bfh')\ll X^{-1}(\log X)^{2s}\sup_{\Gam\in [0,1)}\int_{[0,1)^{l+1}}
|f_{l+1}(\bfalp;2X)^{2s}g_{l+1}(\bfalp,\Gam;X)|\d\bfalp .
\]
Hence, by availing ourselves of Lemma \ref{lemma3.3} we obtain the upper bound
\[
J_{s,k}(X;\bfh)\le J_{s,l+1}(X;\bfh')\ll X^{s-1+\eps}.
\]
The conclusion of Theorem \ref{theorem1.2} follows at once.
\end{proof}

\section{A conditional asymptotic formula, I: minor arcs} Our goal in this section is to 
indicate how, equipped with mean value estimates conjectured to hold that fall short of 
breaking the convexity barrier, one may achieve subconvex minor arc estimates that 
deliver asymptotic formulae for $J_{s,k}(X;\bfh)$ when $s=k(k+1)/2$. Thereby, we prove 
Theorem \ref{theorem1.4}.\par

We begin by extracting from \S8 an estimate sufficient for our purposes. Here, in order to 
simplify our exposition, we introduce some notation. When $A$ is a fixed positive number 
and $0\le j\le k$, we denote by $\grU_j=\grU_j(A)$ the interval 
$\grU_j(A)=[-AX^{-j},AX^{-j}]$. Then, when $0\le m\le k$, we write
\[
\grV_m(A)=\grU_m(A)\times \grU_{m-1}(A)\times \cdots \times \grU_1(A)\times 
[0,1)^{k-m}.
\]
In what follows, we shall refer to Conjecture \ref{conjecture8.1} as the {\it extended main 
conjecture in Vinogradov's mean value theorem}.

\begin{lemma}\label{lemma5.1} Assume the extended main conjecture in Vinogradov's 
mean value theorem. Let $A$ be a fixed positive number. Suppose that $s$ is a positive 
number and $m$ is an integer satisfying $0\le m\le k$. Then, if either
\[
s\ge \tfrac{1}{4}k(k+1)+1\quad \text{or}\quad m(m+1)\le \tfrac{1}{2}k(k+1)-2,
\]
one has
\[
\int_{\grV_m(A)}|f_k(\bfalp;X)|^{2s}\d\bfalp \ll X^\eps ( X^{s-m(m+1)/2}+
X^{2s-k(k+1)/2}).
\]
\end{lemma}

\begin{proof} It follows from the extended main conjecture in Vinogradov's mean value 
theorem that when either
\[
s\ge \tfrac{1}{4}k(k+1)+1\quad \text{or}\quad 
\text{mes}(\grV_m(A))\gg X^{1-k(k+1)/4},
\]
one has
\[
\int_{\grV_m(A)}|f_k(\bfalp;X)|^{2s}\d\bfalp \ll X^\eps (X^s\text{mes}(\grV_m(A))
+X^{2s-k(k+1)/2}).
\]
Such is immediate in the first case from Conjecture \ref{conjecture8.1}, and in the second 
case from Conjecture \ref{conjecture8.2}, which as explained in \S8 is a consequence of 
Conjecture \ref{conjecture8.1}. The upper bound presented in the lemma follows on 
observing that one has $\text{mes}(\grV_m(A))\ll_A X^{-m(m+1)/2}$.
\end{proof}

We first apply this estimate to obtain a bound for a mixed mean value.

\begin{lemma}\label{lemma5.2} Assume the extended main conjecture in Vinogradov's 
mean value theorem, and suppose that $k\ge 2$ and $\bfh\in \dbZ^k\setminus 
\{ {\mathbf 0}\}$. Let $l$ be the smallest index with $l\le k$ for which $h_l\ne 0$, and 
write $v=(k-l)(k-l+1)/2$. Then, provided that $s$ is a natural number with 
$s-v\ge \tfrac{1}{4}k(k+1)+1$, one has
\begin{equation}\label{5.1}
\int_{[0,1)^k}|f_k(\bfalp;2X)^{2s-2v}g_k(\bfalp;X)^{2v}|\d\bfalp \ll 
X^\eps (X^s+X^{2s-k(k+1)/2}).
\end{equation}
\end{lemma}

\begin{proof} It follows from orthogonality that the mean value $I$ on the left hand side of 
(\ref{5.1}) counts the number of integral solutions of the system of equations
\begin{equation}\label{5.2}
\sum_{i=1}^{s-v}(x_i^j-y_i^j)=\sum_{m=1}^v(\nu_j(z_m;\bfh)-\nu_j(t_m;\bfh))\quad 
(1\le j\le k),
\end{equation}
with $1\le \bfx,\bfy\le 2X$ and $1\le \bfz,\bft\le X$. When $1\le j\le k$ and $n_j\in \dbZ$, 
denote by $\rho(\bfn)$ the number of solutions of the system of equations
\[
\sum_{m=1}^v (\nu_j(z_m;\bfh)-\nu_j(t_m;\bfh))=n_j\quad (1\le j\le k),
\]
with $1\le \bfz,\bft\le X$. Then, by orthogonality, one has
\[
\rho(\bfn)=\int_{[0,1)^k}|g_k(\bfalp;X)|^{2v}e(-\bfalp \cdot \bfn)\d\bfalp .
\]
Since the hypotheses of the lemma imply that $h_j=0$ for $1\le j<l$, and $h_l\ne 0$, we 
find from (\ref{2.5}) via the triangle inequality and a change of variables that
\begin{equation}\label{5.3}
\rho(\bfn)\le \int_{[0,1)^k}|g_k(\bfalp;X)|^{2v}\d\bfalp =\int_{[0,1)^{k-l}}
|f_{k-l}(\bfalp;X)|^{2v}\d\bfalp .
\end{equation}
In order to explain the origin of the rightmost mean value in (\ref{5.3}), observe that since 
$h_j=0$ for $1\le j<l$, the polynomial $\nu_j(z;\bfh)$ is non-zero only when $j\ge l$, in 
which case its leading term is $\binom{j}{l}h_lz^{j-l}$. Thus, the first integral in (\ref{5.3}) 
counts the number of integral solutions of the system of equations
\[
\sum_{m=1}^v(z_m^j-t_m^j)=0\quad (1\le j\le k-l),
\]
with $1\le \bfz,\bft\le X$. By orthogonality, the second integral in (\ref{5.3}) counts precisely 
these solutions, justifying the conclusion.\par

We thus have $\rho(\bfn)\le J_{k-l,v}(X;{\mathbf 0})$. Hence, by the (now confirmed) main 
conjecture in Vinogradov's mean value theorem, we deduce that 
$\rho(\bfn)\ll X^{v+\eps}$. We now return to (\ref{5.2}) and note that when 
$1\le \bfz,\bft\le X$, one has
\[
\sum_{m=1}^v(\nu_j(z_m;\bfh)-\nu_j(t_m;\bfh))\ll_\bfh X^{j-l}\quad (l+1\le j\le k)
\]
and
\[
\sum_{m=1}^v(\nu_j(z_m;\bfh)-\nu_j(t_m;\bfh))=0\quad (1\le j\le l).
\]
Hence, in each solution $\bfx,\bfy,\bfz,\bft$ counted by $I$, there are positive numbers 
$C_j(\bfh)$ for which
\begin{equation}\label{5.4}
\biggl| \sum_{i=1}^{s-v}(x_i^j-y_i^j)\biggr| \le C_j(\bfh)X^{j-l}\quad (l+1\le j\le k)
\end{equation}
and
\begin{equation}\label{5.5}
\sum_{i=1}^{s-v}(x_i^j-y_i^j)=0\quad (1\le j\le l).
\end{equation}
Let $I_1$ denote the number of integral solutions of the system (\ref{5.4}) and (\ref{5.5}) 
with $1\le \bfx,\bfy\le 2X$. Then we deduce that
\begin{equation}\label{5.6}
I\le I_1\max_\bfn \rho(\bfn)\ll X^{v+\eps}I_1.
\end{equation}

\par Next we examine the system (\ref{5.4}) and (\ref{5.5}). A standard argument (see for 
example \cite[Lemma 2.1]{Wat1989} shows that
\[
I_1\ll \biggl( \prod_{j=l+1}^kC_j(\bfh)X^{j-l}\biggr) \int_{\grV_{k-l}(A)}
|f_k(\bfalp ;X)|^{2s-2v}\d\bfalp ,
\] 
in which $\grV_{k-l}(A)$ is defined as in the preamble to the statement of Lemma 
\ref{lemma5.1}, and $A=\max \{ C_{l+1}(\bfh),\ldots ,C_k(\bfh)\}$. Thus we deduce that
\[
I_1\ll_\bfh X^v\int_{\grV_{k-l}(A)}|f_k(\bfalp;X)|^{2s-2v}\d\bfalp ,
\]
whence, in view of (\ref{5.6}),
\[
I\ll X^{2v+\eps}\int_{\grV_{k-l}(A)}|f_k(\bfalp;X)|^{2s-2v}\d\bfalp .
\]
We now invoke Lemma \ref{lemma5.1} to bound the mean value on the right hand side 
here. On noting that $\text{mes}(\grV_{k-l}(A))\ll X^{-v}$, we thus deduce that
\[
I\ll X^{2v+\eps}(X^{s-2v}+X^{2s-2v-k(k+1)/2})\ll X^\eps (X^s+X^{2s-k(k+1)/2}).
\]
This completes the proof of the lemma.
\end{proof}

We apply this estimate in combination with Lemma \ref{lemma2.1} to obtain an acceptable 
minor arc bound of use in our application of the Hardy-Littlewood method. When $Q$ is a 
real parameter with $1\le Q\le X$, we define the set of major arcs $\grM(Q)$ to be the 
union of the arcs
\[
\grM(q,a)=\{ \alp\in [0,1):|q\alp -a|\le QX^{-k}\},
\]
with $0\le a\le q\le Q$ and $(a,q)=1$. We then define the complementary set of minor arcs 
$\grm(Q)=[0,1)\setminus \grM(Q)$.

\begin{lemma}\label{lemma5.3} Assume the extended main conjecture in Vinogradov's 
mean value theorem, and suppose that $k\ge 3$ and $\bfh\in \dbZ^k$. Let $s$ be a 
natural number with $s\ge k(k+1)/2$, and put $\del=2/(k^2(k-1)^2)$. Then provided that 
$h_l\ne 0$ for some index $l$ with $1\le l<k$, one has
\[
I_s(\grm(Q);X;\bfh)\ll X^{2s-\frac{1}{2}k(k+1)+\eps}Q^{-\del}.
\]
\end{lemma}

\begin{proof} We find from Lemma \ref{lemma2.1} that
\begin{equation}\label{5.7}
I_s(\grm(Q);X;\bfh)\ll X^{\eps-1}\sup_{\Gam\in [0,1)}V_s(X;\bfh),
\end{equation}
where
\[
V_s(X;\bfh)=\int_{\grm(Q)}\int_{[0,1)^{k-1}}|f_k(\bfalp;2X)^{2s}g_k(\bfalp,\Gam;X)|
\d\bfalp .
\]
Write $v=(k-l)(k-l+1)/2$ and $u=k(k+1)/2$, and put
\[
\ome_0=2s-k(k+1)+1/v,\quad \ome_1=1-1/(2v)\quad \text{and}\quad \ome_2=1/(2v).
\]
Then for $s\ge k(k+1)/2$, an application of H\"older's inequality yields the bound
\begin{equation}\label{5.8}
V_s(X;\bfh)\le V_0^{\ome_0}V_1^{\ome_1}V_2^{\ome_2},
\end{equation}
where
\begin{align*}
V_0&=\sup_{\bfalp\in [0,1)^{k-1}\times \grm(Q)}|f_k(\bfalp;2X)|,\\
V_1&=\int_{[0,1)^k}|f_k(\bfalp;2X)|^{2u}\d\bfalp ,\\
V_2&=\int_{[0,1)^k}|f_k(\bfalp;2X)^{2u-2}g_k(\bfalp,\Gam;X)^{2v}|\d\bfalp .
\end{align*}

\par It follows from \cite[Lemma 2.2]{Woo2021b} that $V_0\ll X^{1+\eps}Q^{-\sig}$, 
where $\sig=1/(k(k-1))$. Also, from the (now confirmed) main conjecture in Vinogradov's 
mean value theorem, we have $V_1\ll X^{u+\eps}$, whilst Lemma \ref{lemma5.2} delivers 
the conditional bound $V_2\ll X^{2v+u-2+\eps}$ by the now familiar routine. Thus we 
deduce from (\ref{5.8}) that
\begin{equation}\label{5.9}
V_s(X;\bfh)\ll X^{2s+1-u+\eps}Q^{-\ome_0\sig}\ll 
X^{2s+1-\frac{1}{2}k(k+1)+\eps}Q^{-\del},
\end{equation}
where
\[
\del=\frac{1}{vk(k-1)}\ge \frac{2}{k^2(k-1)^2}.
\]
The proof of the lemma is completed by substituting (\ref{5.9}) into (\ref{5.7}).
\end{proof}

\section{A conditional asymptotic formula, II: the endgame} Equipped now with the 
conditional minor arc estimate supplied by Lemma \ref{lemma5.3}, our proof of Theorem 
\ref{theorem1.4} follows the argument applied in our previous work 
\cite[\S\S5-7]{Woo2021c} concerning the cubic case of the inhomogeneous Vinogradov 
system. There are few if any complications. When $\grA\subseteq [0,1)^k$ is measurable, 
we define the mean value $T_s(\grA)=T_s(\grA;X;\bfh)$ by putting
\begin{equation}\label{6.1}
T_s(\grA;X;\bfh)=\int_\grA |f_k(\bfalp;X)|^{2s}e(-\bfalp \cdot \bfh)\d\bfalp .
\end{equation}

\par We formulate Hardy-Littlewood dissections of the unit cube $[0,1)^k$ suitable for our 
purpose. When $Z$ is a real parameter with $1\le Z\le X$, we define the set of major arcs 
$\grK(Z)$ to be the union of the arcs
\[
\grK(q,\bfa;Z)=\{\bfalp \in [0,1)^k:\text{$|\alp_j-a_j/q|\le ZX^{-j}$ $(1\le j\le k)$}\},
\]
with $1\le q\le Z$, $0\le a_j\le q$ $(1\le j\le k)$ and $(q,a_1,\ldots ,a_k)=1$. We then 
define the complementary set of minor arcs $\grk(Z)=[0,1)^k\setminus \grK(Z)$.\par

Recall the one-dimensional Hardy-Littlewood dissection of $[0,1)$ into sets of major arcs 
$\grM(Q)$ and minor arcs $\grm(Q)$ introduced in the preamble to Lemma 
\ref{lemma5.3}. We now fix $L=X^{1/(8k^2)}$ and $Q=L^k$, and we define a 
$k$-dimensional set of arcs by taking $\grN=\grK(Q^2)$ and $\grn=\grk(Q^2)$. This 
intermediate Hardy-Littlewood dissection may be refined to obtain the narrow set of major 
arcs $\grP=\grK(L)$ and the corresponding set of minor arcs $\grp=\grk(L)$. It is useful 
then to write $\grP(q,\bfa)=\grK(q,\bfa;L)$. One readily confirms that $\grP\subseteq 
[0,1)^{k-1}\times \grM$, and hence the set of points $(\alp_1,\ldots ,\alp_k)$ lying in 
$[0,1)^k$ may be partitioned into the four disjoint subsets
\begin{align*}
\grW_1&=[0,1)^{k-1}\times \grm,\\
\grW_2&=\left( [0,1)^{k-1}\times \grM\right)\cap \grn,\\
\grW_3&=\left( [0,1)^{k-1}\times \grM\right) \cap (\grN\setminus \grP),\\
\grW_4&=\grP.
\end{align*}
Thus, in view of (\ref{2.3}) and (\ref{6.1}), one sees that
\begin{equation}\label{6.2}
J_{s,k}(X;\bfh)=T_s([0,1)^k)=\sum_{i=1}^4T_s(\grW_i).
\end{equation}

\par We assume throughout the extended main conjecture in Vinogradov's mean value 
theorem. Then by substituting $Q=X^{1/(8k)}$ into Lemma \ref{lemma5.3}, we deduce 
that when $\bfh\in \dbZ^k$ and $h_l\ne 0$ for some index $l$ with $1\le l<k$, one has
\begin{equation}\label{6.3}
T_s(\grW_1)=I_s(\grm(Q);X;\bfh)\ll X^{2s-\frac{1}{2}k(k+1)-1/(4k^5)}.
\end{equation}
Our definition of the sets of arcs $\grN$, $\grn$, $\grP$ and $\grp$ in the present memoir 
is identical with that employed in \cite[\S\S3-6]{Woo2021b}. Thus, the analysis applied in 
\cite[\S4]{Woo2021b} may be employed without material alteration in present 
circumstances to obtain the upper bound
\begin{equation}\label{6.4}
T_s(\grW_2)\ll X^{2s-\frac{1}{2}k(k+1)-1/(16k)}.
\end{equation}
Likewise, the analysis of \cite[\S5]{Woo2021b} applies, mutatis mutandis, to reveal that
\begin{equation}\label{6.5}
T_s(\grW_3)\ll X^{2s-\frac{1}{2}k(k+1)-1/(12k^3)}.
\end{equation}
Finally, the discussion of \cite[\S6]{Woo2021b} provides a template for the analysis of the 
major arcs in the present circumstances that may be applied almost without modification. 
Recall the definitions (\ref{1.4}) and (\ref{1.5}), and in addition the notational device of 
writing $n_j=h_jX^{-j}$ $(1\le j\le k)$. Then one finds that when 
$2s>\tfrac{1}{2}k(k+1)+2$, one has
\begin{equation}\label{6.6}
T_s(\grW_4)=T_s(\grP)=\grS_{s,k}(\bfh)\grJ_{s,k}(\bfh)X^{2s-k(k+1)/2}+
o(X^{2s-k(k+1)/2}),
\end{equation}
where
\begin{equation}\label{6.7}
\grJ_{s,k}(\bfh)=\int_{\dbR^k}|I(\bfbet)|^{2s}e(-\bfbet \cdot \bfn)\d\bfbet
\end{equation}
and
\begin{equation}\label{6.8}
\grS_{s,k}(\bfh)=\sum_{q=1}^\infty \sum_{\substack{1\le \bfa\le q\\ 
(q,a_1,\ldots ,a_k)=1}}|q^{-1}S(q,\bfa)|^{2s}e_q(-\bfa \cdot \bfh).
\end{equation}
By substituting the relations (\ref{6.3}) to (\ref{6.6}) into (\ref{6.2}), we conclude that
\[
J_{s,k}(X;\bfh)=\grS_{s,k}(\bfh)\grJ_{s,k}(\bfh)X^{2s-k(k+1)/2}+o(X^{2s-k(k+1)/2}).
\]
Take $s=k(k+1)/2$, and recall the notation (\ref{1.6}) and (\ref{1.7}). Then we obtain the 
relation
\[
B_k(X;\bfh)=\grS_k(\bfh)\grJ_k(\bfh)X^{k(k+1)/2}+o(X^{k(k+1)/2}),
\]
confirming the principal conclusion of Theorem \ref{theorem1.4}.\par

The observation that $0\le \grJ_k(\bfh)\ll 1$ and $0\le \grS_k(\bfh)\ll 1$ follows from 
the absolute convergence of $\grS_{s,k}(\bfh)$ and $\grJ_{s,k}(\bfh)$ when 
$2s>\tfrac{1}{2}k(k+1)+2$, combined with the standard theory of the singular series and 
singular integral described in \cite[Theorem 3.7]{AKC2004}. This completes the proof of 
Theorem \ref{theorem1.4}.

\section{Paucity and subconvexity} Our objective in this section is, not only to establish 
Theorem \ref{theorem1.5}, but also to illustrate the role that paucity phenomena play in 
subconvexity results associated with inhomogeneous Vinogradov systems. This circle of 
ideas is relevant in investigations of $J_{s,k}(X;\bfh)$ when $s$ is small, which is to say, 
no larger than $k$ or thereabouts. We begin with two almost trivial observations. The 
first shows that one cannot in general expect to obtain upper bounds in which one saves 
more than a factor $X$ over the convexity limited estimates exhibiting square-root 
cancellation.

\begin{theorem}\label{theorem7.1} Suppose that $s,k\in \dbN$. Then
\[
\max_{\bfh\in \dbZ^k\setminus \{{\mathbf 0}\}}J_{s,k}(X;\bfh)\gg X^{s-1}.
\]
\end{theorem}

\begin{proof} We fix integers $a$ and $b$ with $a>b$, say $a=2$ and $b=1$, and then 
fix $h_j=a^j-b^j$ $(1\le j\le k)$. Thus we have $\bfh\ne {\mathbf 0}$ and, for any 
$(s-1)$-tuple $\bft$ with $1\le \bft\le X$, the system (\ref{1.1}) has the solution
\[
\bfx=(t_1,\ldots ,t_{s-1},a),\quad \bfy=(t_1,\ldots ,t_{s-1},b).
\]
In this way, we see that when $X\in \dbN$ one has $J_{s,k}(X;\bfh)\ge X^{s-1}$.
\end{proof}

This theorem shows that the conclusion of Theorem \ref{theorem1.2} may be regarded as 
close to best possible. Moreover, any improvement in the upper bound 
$J_{s,k}(X;\bfh)\ll X^{s-1}$ must account for the special subvarieties of the complete 
intersection defined by (\ref{1.1}) containing subdiagonal solutions, in the appropriate 
sense. The most extreme such situation is addressed in the second of these almost trivial 
conclusions.\par

\begin{theorem}\label{theorem7.2} Suppose that $s,k\in \dbN$ and 
$\bfh\in \dbZ^k\setminus \{{\mathbf 0}\}$. Suppose that $h_j=0$ for precisely $t$ indices, 
say $j\in \{j_1,\ldots ,j_t\}$ with $1\le j_1<j_2<\ldots <j_t\le k$. Then provided that 
$1\le s\le t$, one has $J_{s,k}(X;\bfh)=0$.
\end{theorem}

\begin{proof} The hypothesis on $\bfh$ in the statement of the theorem ensures that 
whenever $\bfx$ and $\bfy$ satisfy the system (\ref{1.1}), then one has
\begin{equation}\label{7.1}
\sum_{i=1}^sx_i^{j_l}=\sum_{i=1}^sy_i^{j_l}\quad (1\le l\le t).
\end{equation}
When $\bfx,\bfy\in \dbN^s$ and $1\le s\le t$, it follows from \cite{Ste1971} that in all 
solutions of the system (\ref{7.1}), the $s$-tuple $(x_1,\ldots ,x_s)$ is a permutation of 
$(y_1,\ldots ,y_s)$. For any such solution, we find from (\ref{1.1}) that 
$\bfh={\mathbf 0}$, contradicting the hypothesis from the statement of the theorem. Thus 
we conclude that $J_{s,k}(X;\bfh)=0$.
\end{proof}

We now turn to the proof of Theorem \ref{theorem1.5}. This we view as establishing the 
principle that when there are few indices $l$ for which $h_l\ne 0$, then $J_{k,k}(X;\bfh)$ 
may be expected to be very small, and indeed far smaller than would be implied by the 
convexity limited bound. Our proof of this conclusion makes heavy use of our earlier work 
on paucity in relatives of Vinogradov systems \cite{Woo2021a}.

\begin{proof}[The proof of Theorem \ref{theorem1.5}] We work under the hypotheses of 
the statement of the theorem. Let $J^*_{k,k}(X;\bfh)$ denote the number of solutions of 
the system (\ref{1.1}) with $s=k$ in which $x_i=y_m$ for no indices $i$ and $m$ with 
$1\le i,m\le t$. Consider a solution $\bfx,\bfy$ of the system (\ref{1.1}) with $s=k$ counted 
by $J_{k,k}(X;\bfh)$ but not by $J_{k,k}^*(X;\bfh)$. By relabelling variables, we may 
suppose that $x_k=y_k$, and hence we deduce from (\ref{1.1}) that
\[
J_{k,k}(X;\bfh)-J_{k,k}^*(X;\bfh)\ll XJ_{k-1,k}(X;\bfh).
\]
However, since $h_j=0$ for the $k-1$ indices $j$ with $j\ne l$, it is apparent from Theorem 
\ref{theorem7.2} that $J_{k-1,k}(X;\bfh)=0$. Thus we conclude that
\[
J_{k,k}(X;\bfh)=J_{k,k}^*(X;\bfh).
\]

Consider next a solution $\bfx,\bfy$ of (\ref{1.1}) with $s=k$ counted by 
$J^*_{k,k}(X;\bfh)$. Define the elementary symmetric polynomials 
$\sig_j(\bfz)\in \dbZ[z_1,\ldots ,z_k]$ via the generating function identity
\[
\sum_{j=0}^k\sig_j(\bfz)t^j=\prod_{i=1}^k(1+tz_i),
\]
and write, further, 
\[
s_j(\bfz)=z_1^j+\ldots +z_k^j\quad (1\le j\le k).
\]
When $n\ge 1$, a familiar formula (see \cite[equation (2.2)]{Woo2021a}) delivers the 
relation
\[
\sig_n(\bfz)=(-1)^n\sum_{\substack{m_1+2m_2+\ldots +nm_n=n\\ m_i\ge 0}}
\prod_{i=1}^n\frac{(-s_i(\bfz))^{m_i}}{i^{m_i}m_i!}.
\]
The system of equations
\[
s_j(\bfx)=s_j(\bfy)+h_j\quad (1\le j\le k),
\]
is tantamount to (\ref{1.1}) with $s=k$. Since $h_j=0$ for $j\ne l$, we deduce that
\[
\sig_n(\bfx)=(-1)^n\sum_{\substack{m_1+2m_2+\ldots +nm_n=n\\ m_i\ge 0}}
\biggl( \frac{-s_l(\bfy)-h_l}{l^{m_l}m_l!}\biggr)^{m_l}
\prod_{\substack{1\le i\le n\\ i\ne l}}\frac{(-s_i(\bfy))^{m_i}}{i^{m_i}m_i!}.
\]
We conclude that
\begin{equation}\label{7.2}
\sig_n(\bfx)=\sig_n(\bfy)\quad (1\le n<l),
\end{equation}
and, when $n\ge l$, that there is a weighted homogeneous polynomial $\Psi_n(h;\bfy)$ 
having rational coefficients and satisfying the property that
\begin{equation}\label{7.3}
\sig_n(\bfx)-\sig_n(\bfy)=h_l\Psi_n(h_l;\bfy).
\end{equation}
Here, if a monomial term of $\Psi_n(h;\bfy)$ has total degree $d$ in terms of $\bfy$ and 
degree $e$ in terms of $h$, then one has $d+le=n-l$. It is evident, moreover, that there is 
a non-zero integer $A_n$, with $A_n=O_k(1)$, having the property that $A_n\Psi_n(h;\bfy)$ 
has integer coefficients all of size $O_k(1)$. In particular, in each solution $\bfx,\bfy$ of 
(\ref{1.1}) counted by $J^*_{k,k}(X;\bfh)$, we have $|\Psi_n(h_l;\bfy)|\ll X^{n-l}$.\par

We deduce from (\ref{7.2}) and (\ref{7.3}) that for the indeterminate $z$, one has
\begin{align*}
\prod_{i=1}^k(z-x_i)-\prod_{i=1}^k(z-y_i)&=(-1)^k\sum_{n=0}^k
(\sig_n(\bfx)-\sig_n(\bfy))(-z)^{k-n}\\
&=(-1)^kh_l\sum_{n=l}^k\Psi_n(h_l;\bfy)(-z)^{k-n}.
\end{align*}
By substituting $z=y_j$, we therefore infer that there is a polynomial $\tau(\bfy;y_j;h)$ for 
which one has
\[
A_1\ldots A_k\prod_{i=1}^k(y_j-x_i)=h_l\tau(\bfy;y_j;h_l).
\]
This polynomial $\tau(\bfy;y_j;h)$ has integer coefficients and is weighted homogeneous of 
total degree $k-l$, with each variable $y_i$ carrying weight $1$, and the variable $h$ 
carrying weight $l$. In particular, with the choices for $\bfy,y_j,h_l$ associated with the 
solution $\bfx,\bfy$ counted by $J^*_{k,k}(X;\bfh)$ currently under consideration, we may 
assume that $\tau(\bfy;y_j;h_l)$ is an integer of size $O(X^{k-l})$.\par

There are $O(X^{k-l})$ possible choices for the integer $\tau(\bfy;y_j;h_l)$, and hence also 
for $h_l\tau(\bfy;y_j;h_l)$. None of these are zero, for this would contradict the 
non-vanishing of $y_j-x_i$ $(1\le i,j\le k)$. Then, for each of the $O(X^{k-l})$ possible 
choices for $N(\bfy)=h_l\tau(\bfy;y_j;h_l)$, we see that the integers $y_j-x_i$ $(1\le i\le k)$ 
are divisors of $N(\bfy)$. Since $N(\bfy)=O(X^k)$, a standard divisor function estimate 
reveals that there are $O(X^\eps)$ possible choices for these divisors. Fixing any one of 
these choices and one of the $O(X)$ possible choices for $y_j$, it follows that 
$x_1,\ldots ,x_k$ and $y_j$ are now all fixed. Next, interchanging the roles of $\bfx$ and 
$\bfy$ in the argument just described, we find that
\[
A_1\ldots A_k\prod_{i=1}^k(x_j-y_i)=-h_l\tau(\bfx;x_j;-h_l).
\]
Since $\bfx$ is already fixed, it follows that the integers $x_j-y_i$ are all divisors of the 
fixed non-zero integer $N'(\bfx)=h_l\tau(\bfx,x_j;-h_l)$. As in the situation just discussed, 
there are $O(X^\eps)$ possible choices for these divisors. Fixing any one of these choices, 
and noting that $x_j$ is already fixed, we find that the integers $y_1,\ldots ,y_k$ are now 
also fixed. The total number of choices for $\bfx$ and $\bfy$ is consequently 
$O(X^{k-l+1+\eps})$. This confirms that 
\[
J_{k,k}(X;\bfh)=J^*_{k,k}(X;\bfh)\ll X^{k-l+1+\eps}
\]
and completes the proof of the theorem.
\end{proof}

\section{Appendix: the extended main conjecture in Vinogradov's mean value theorem} 
The purpose of this section is to discuss an extension to the main conjecture in Vinogradov's 
mean value theorem previously announced in 2019 by the author at a workshop in 
Oberwolfach. Since this conjecture offers numerous consequences, including but not limited 
to Theorem \ref{theorem1.4}, we take the opportunity to discuss its origin, nature, and 
its implications relevant herein.\par

We begin by recalling the main conjecture in Vinogradov's mean value theorem, proved in 
work of Bourgain, Demeter and Guth \cite{BDG2016} and in work of the author 
\cite{Woo2016, Woo2019}. A brief account of the history of Vinogradov's mean value 
theorem, and developments at the cusp of the proof of the main conjecture, is offered in 
\cite{Woo2014}. For the present discussion, we write $J_{s,k}(X)$ for 
$J_{s,k}(X;{\mathbf 0})$, which counts the number of integral solutions of the system of 
equations
\[
\sum_{i=1}^s(x_i^j-y_i^j)=0\quad (1\le j\le k),
\]
with $1\le \bfx,\bfy\le X$. When $\grB\subseteq [0,1)^k$ is measurable and $s>0$, we 
write
\begin{equation}\label{8.1}
M_{s,k}(X;\grB)=\int_\grB |f_k(\bfalp;X)|^{2s}\d\bfalp .
\end{equation}
Thus, by orthogonality, when $s\in \dbN$ we have $J_{s,k}(X)=M_{s,k}(X;[0,1)^k)$. The 
main conjecture in Vinogradov's mean value theorem asserts that
\begin{equation}\label{8.2}
J_{s,k}(X)\ll X^{s+\eps}+X^{2s-k(k+1)/2}.
\end{equation}

\par In order to motivate the formulation of the extended main conjecture, we briefly 
sketch how the two terms on the right hand side of (\ref{8.2}) arise from the application of 
the circle method. Consider a Hardy-Littlewood dissection of the unit cube $[0,1)^k$ into 
sets of major and minor arcs of the type $\grK(Y)$ and $\grk(Y)$, for a suitable parameter 
$Y$, as defined in \S6. Define the generating functions $I(\bfbet)$ and $S(q,\bfa)$ as in 
(\ref{1.4}) and (\ref{1.5}), and put
\begin{align*}
I(\bfbet;X)&=\int_0^Xe(\bet_1\gam+\bet_2\gam^2+\ldots +\bet_k\gam^k)\d\gam\\
&=XI(\bet_1X,\bet_2X^2,\ldots ,\bet_kX^k).
\end{align*}
When $\bfalp \in \grK(q,\bfa;Y)\subseteq \grK(Y)$, the exponential sum $f_k(\bfalp;X)$ is 
closely approximated by $q^{-1}S(q,\bfa)I(\bfalp-\bfa/q;X)$. Until recently, it was widely 
believed by many experts in the Hardy-Littlewood method that when $\bfalp\in [0,1)^k$, 
there should exist $q\in \dbN$ and $\bfa\in \dbZ^k$ with $(q,a_1,\ldots ,a_k)=1$ satisfying
\begin{equation}\label{8.3}
f_k(\bfalp;X)-q^{-1}S(q,\bfa)I(\bfalp -\bfa/q;X)\ll X^{1/2+\eps}.
\end{equation}
Recent work of Brandes et al.~\cite{BPPSV2021} shows that such a strong relation cannot 
be true in full generality when $\bfalp$ is very close to $\bfa/q$. However, any failure of 
this relation is expected to produce a small number of secondary terms also of major arc 
type, and hence is not expected to have any material impact on the outcome of the ensuing 
discussion.\par

Recalling the definitions (\ref{6.7}) and (\ref{6.8}) of $\grJ_{s,k}(\bfh)$ and 
$\grS_{s,k}(\bfh)$, one finds that the contribution of the term 
$q^{-1}S(q,\bfa)I(\bfalp-\bfa/q;X)$ from (\ref{8.3}) within the mean value $J_{s,k}(X)$ is 
at most $X^{2s-k(k+1)/2}I_0$, where
\[
I_0=\sum_{q=1}^\infty 
\sum_{\substack{1\le a_1,\ldots ,a_k\le q\\ (q,a_1,\ldots ,a_k)=1}}
|q^{-1}S(q,\bfa)|^{2s}\int_{-\infty}^\infty |I(\bfbet)|^{2s}\d\bfbet 
=\grS_{s,k}({\mathbf 0})\grJ_{s,k}({\mathbf 0}).
\]
The details of the argument here are familiar from the analysis of the major arc 
contribution in an application of the circle method to the problem. We note that the singular 
series $\grS_{s,k}({\mathbf 0})$ converges absolutely for $2s>\tfrac{1}{2}k(k+1)+2$, and 
the singular integral $\grJ_{s,k}({\mathbf 0})$ converges absolutely for 
$2s>\tfrac{1}{2}k(k+1)+1$ (see \cite[Theorem 1]{Ark1984}). Thus, under the first of 
these conditions on $s$, we obtain
\[
J_{s,k}(X)=\int_{[0,1)^k}|f_k(\bfalp;X)|^{2s}\d\bfalp \ll 
X^{2s-k(k+1)/2}+\int_{[0,1)^k}(X^{1/2+\eps})^{2s}\d\bfalp ,
\]
and thereby we recover the main conjecture (\ref{8.2}) established in 
\cite{BDG2016, Woo2016, Woo2019}. For smaller values of $s$, the same conclusion 
follows by application of H\"older's inequality, since the term $X^{s+\eps}$ dominates, so 
(\ref{8.2}) follows for all $s\in \dbN$. As a final comment relevant to this preliminary 
discussion, we remark that the aforementioned deviations from this model suggested by 
work of \cite{BPPSV2021} would, at worst, inflate the above estimates by a factor of 
$X^\eps$ for smaller values of $s$, and this has no impact in our wider discusion.\par

The question now arises concerning what outcome is to be expected when we integrate, 
not over the whole unit cube $[0,1)^k$, but instead over a subset $\grB$. The same 
philosophy demonstrates that when $s$ is any real number satisfying 
$2s>\tfrac{1}{2}k(k+1)+2$, one should have the estimate
\[
\int_\grB |f_k(\bfalp;X)|^{2s}\d\bfalp \ll X^{2s-k(k+1)/2}I_0+
\int_\grB (X^{1/2+\eps})^{2s}\d\bfalp ,
\]
whence, on recalling the definition (\ref{8.1}),
\[
M_{s,k}(X;\grB)\ll X^{2s-k(k+1)/2}+X^{s+\eps}\text{mes}(\grB).
\]
This is tantamount to the extended main conjecture in Vinogradov's mean value theorem.

\begin{conjecture}\label{conjecture8.1} Suppose that $k\in \dbN$ and 
$\grB\subseteq [0,1)^k$ is measurable. Then whenever $s$ is a real number with 
$s\ge \tfrac{1}{4}k(k+1)+1$, one has
\[
\int_\grB |f_k(\bfalp;X)|^{2s}\d\bfalp \ll X^\eps \left( X^s\text{mes}(\grB)+
X^{2s-k(k+1)/2}\right).
\]
\end{conjecture}

We have limited the values of $s$ admissible in this conclusion to the range 
$s\ge \tfrac{1}{4}k(k+1)+1$ in order that appropriate convergence of the singular series 
$\grS_{s,k}({\mathbf 0})$ and singular integral $\grJ_{s,k}({\mathbf 0})$ be guaranteed. 
For larger values of $s$ absolute convergence follows from \cite[Theorem 1]{Ark1984}, and 
an application of H\"older's inequality delivers similar conclusions at the cost of inflating 
bounds by a factor of $X^\eps$ when $s=\tfrac{1}{4}k(k+1)+1$. For smaller values of $s$, 
the conjecture should be modified to reflect a larger secondary term arising from the 
potential divergence of these quantities. However, one may recover a cheap but useful 
version of the conjecture applicable for all $s$ provided that $\text{mes}(\grB)$ is not too 
small.

\begin{conjecture}\label{conjecture8.2} Suppose that $k\in \dbN$ and 
$\grB\subseteq [0,1)^k$ is measurable. Then whenever $s$ is a positive number and
\begin{equation}\label{8.4}
\text{mes}(\grB)\gg X^{1-k(k+1)/4},
\end{equation}
one has
\begin{equation}\label{8.5}
\int_\grB |f_k(\bfalp;X)|^{2s}\d\bfalp \ll X^\eps \left( X^s\text{mes}(\grB)+
X^{2s-k(k+1)/2}\right) .
\end{equation}
\end{conjecture}

To see that Conjecture \ref{conjecture8.2} follows from Conjecture \ref{conjecture8.1}, 
note first that the conclusion (\ref{8.5}) of the former is immediate from the latter in the 
situation wherein $s\ge\tfrac{1}{4}k(k+1)+1$. Suppose then that $t$ is a positive number 
with $t<\tfrac{1}{4}k(k+1)+1$, and put $K=\tfrac{1}{2}k(k+1)+2$. Then on assuming the 
validity of Conjecture \ref{conjecture8.1}, an application of H\"older's inequality yields
\begin{align*}
\int_\grB |f_k(\bfalp;X)|^{2t}\d\bfalp &\le \biggl( \int_\grB\d\bfalp \biggr)^{1-2t/K}
\biggl( \int_\grB |f_k(\bfalp;X)|^K\d\bfalp \biggr)^{2t/K}\\
&\ll X^\eps \left( \text{mes}(\grB)\right)^{1-2t/K}\left( X^{K/2}\text{mes}(\grB)+
X^{K-k(k+1)/2}\right)^{2t/K}\\
&\ll X^{\eps}\left( X^t\text{mes}(\grB)+(\text{mes}(\grB))^{1-2t/K}X^{4t/K}\right) .
\end{align*}
The second term here is asymptotically majorised by the first provided that 
$(\text{mes}(\grB))^{2t/K}\gg X^{4t/K-t}$, a condition that is satisfied when 
$\text{mes}(\grB)\gg X^{2-K/2}$. The conclusion (\ref{8.5}) therefore follows provided 
that (\ref{8.4}) holds.\par

We next consider some consequences and limitations of these conjectures.

\begin{theorem}\label{theorem8.3} Assume the extended main conjecture in Vinogradov's 
mean value theorem. Consider positive numbers $\tet_1,\ldots ,\tet_k$ and the box
\[
\grB(\bftet)=[-X^{-\tet_1},X^{-\tet_1}]\times \cdots \times [-X^{-\tet_k},X^{-\tet_k}].
\]
Then whenever $s\ge \tfrac{1}{4}k(k+1)+1$, one has
\[
\int_{\grB(\bftet)}|f_k(\bfalp;X)|^{2s}\d\bfalp \ll X^\eps \left( 
X^{s-\tet_1-\ldots -\tet_k}+X^{2s-k(k+1)/2}\right).
\]
The same conclusion also holds without condition on $s$ provided that
\[
\tet_1+\ldots +\tet_k\le \tfrac{1}{4}k(k+1)-1.
\]
\end{theorem}

\begin{proof} The conclusions of the theorem are immediate from Conjectures 
\ref{conjecture8.1} and \ref{conjecture8.2} on observing that 
$\text{mes}(\grB(\bftet))\asymp X^{-\tet_1-\ldots -\tet_k}$.
\end{proof}

The cap sets $\grB=[0,1]^{k-1}\times [0,X^{-\tet_k}]$ considered by Demeter, Guth and 
Wang \cite{DGW2020} are addressed by the special case $\bftet=(0,\ldots ,0,\tet_k)$ of 
this theorem. Indeed, the reader will find that \cite[Conjecture 2.5]{DGW2020} asserts that 
the conclusion of Theorem \ref{theorem8.3} should hold in the special case 
$\grB(\bftet)=[0,1]^{k-1}\times [0,X^{-\tet_k}]$, provided that $0\le \tet_k\le k-1$ and 
$k\ge 2$. These authors proved this conjecture when $k=3$ and $0\le \tet_3\le 3/2$ in 
the special case $s=6-\tet_3$ (see \cite[Theorem 3.3]{DGW2020}). Plainly, the (conditional) 
conclusion of Theorem \ref{theorem8.3} is decidedly more general in scope, and Conjecture 
\ref{conjecture8.1} is of even wider generality.\par

We remark that in the situation wherein $\grB$ is restricted to the kind of generalised cap 
sets that are the subject of Theorem \ref{theorem8.3}, the major arc analysis implicit in 
the formulation of the conjecture may presumably be improved. The absolute convergence 
of the singular series $\grS_{s,k}({\mathbf 0})$ requires that one have 
$s>\tfrac{1}{4}k(k+1)+1$ owing in part to the extra divergence arising from the sum over 
$a_1,\ldots ,a_k$ implicit in the definition (\ref{6.8}). If one or more of the variables 
$\alp_j$ is restricted in a manner ensuring that $a_j$ is limited to a much smaller range 
than the interval $[1,q]$, as is the case in cap set problems, then presumably there is scope 
for improving the condition on $s$ towards the less onerous constraint 
$s>\tfrac{1}{4}k(k+1)+\tfrac{1}{2}$.\par

We finish this appendix by noting that the conclusion of Conjecture \ref{conjecture8.1} 
cannot hold as stated when $s<\tfrac{1}{4}k(k+1)$. In order to confirm this assertion, 
consider the system of inequalities
\begin{equation}\label{8.6}
\biggl| \sum_{i=1}^s(x_i^j-y_i^j)\biggr|\le sX^{j-1}\quad (1\le j\le k).
\end{equation}
Denote by $\Ome_1(X)$ the number of solutions of the system of inequalities (\ref{8.6}) 
with $1\le \bfx,\bfy\le X$. We observe that when $1\le \bfx,\bfy\le X^{1-1/k}$, then the 
system (\ref{8.6}) is satisfied whenever
\begin{equation}\label{8.7}
\biggl| \sum_{i=1}^s(x_i^j-y_i^j)\biggr|\le s(X^{1-1/k})^jX^{-(k-j)/k}\quad (1\le j\le k-1).
\end{equation}
Notice, in particular, that the inequality in (\ref{8.6}) corresponding to the exponent $j=k$ 
is automatically satisfied in these circumstances. We denote by $\Ome_2(X)$ the number of 
solutions of (\ref{8.7}) subject to this condition $1\le \bfx,\bfy\le X^{1-1/k}$. Thus we 
have the lower bound $\Ome_1(X)\ge \Ome_2(X)$.\par

Write
\[
\grD=\bigtimes_{j=1}^{k-1}\Bigl[ -\frac{1}{2s}(X^{1-1/k})^{-j}X^{(k-j)/k},\frac{1}{2s}
(X^{1-1/k})^{-j}X^{(k-j)/k}\Big].
\]
Then a standard argument (see for example \cite[Lemma 2.1]{Wat1989}) shows that
\begin{equation}\label{8.8}
\Ome_2(X)\gg \biggl( \prod_{j=1}^{k-1}(X^{1-1/k})^jX^{-(k-j)/k}\biggr) \int_\grD 
|f_{k-1}(\bfalp;X^{1-1/k})|^{2s}\d\bfalp .
\end{equation}
Next define a narrow major arc around ${\mathbf 0}$ by taking $\tau>0$ sufficiently small 
in terms of $s$ and $k$, and put
\[
\grD_0=\bigtimes_{j=1}^{k-1}\Bigl[ -\tau(X^{1-1/k})^{-j},\tau (X^{1-1/k})^{-j}\Big].
\]
Standard arguments from the theory of Vinogradov's mean value theorem (see 
\cite[Chapter 7]{Vau1997}) show that for $\bfalp\in \grD_0$ one has 
$|f_{k-1}(\bfalp;X^{1-1/k})|\gg X^{1-1/k}$, whence
\begin{align*}
\int_{\grD_0}|f_{k-1}(\bfalp;X^{1-1/k})|^{2s}\d\bfalp &\gg (X^{1-1/k})^{2s}
\text{mes}(\grD_0)\\
&\gg (X^{1-1/k})^{2s-k(k-1)/2}.
\end{align*}
Since we may assume that $\grD_0\subseteq \grD$, we deduce from (\ref{8.8}) that
\begin{align*}
\Ome_2(X)&\gg X^{(k-1)(k-2)/2}\int_{\grD_0}|f_{k-1}(\bfalp;X^{1-1/k})|^{2s}\d\bfalp \\
&\gg X^{(k-1)(k-2)/2}\cdot X^{2s(1-1/k)-(k-1)^2/2}.
\end{align*}
Thus, the lower bound $\Ome_1(X)\ge \Ome_2(X)$ leads us to the conclusion
\begin{equation}\label{8.9}
\Ome_1(X)\gg X^{2s(1-1/k)-(k-1)/2}.
\end{equation}

\par On the other hand, again employing the same standard argument (see 
\cite[Lemma 2.1]{Wat1989}), we find that
\[
\Ome_1(X)\ll \biggl( \prod_{j=1}^kX^{j-1}\biggr)\int_{\grD_1}|f_k(\bfalp;X)|^{2s}
\d\bfalp ,
\]
where
\[
\grD_1=\bigtimes_{i=1}^k\Bigl[ -\frac{1}{2s}X^{1-j},\frac{1}{2s}X^{1-j}\Bigr] .
\]
Thus
\[
\Ome_1(X)\ll X^{k(k-1)/2}\int_{\grD_1}|f_k(\bfalp;X)|^{2s}\d\bfalp .
\]
Assuming the validity of Conjecture \ref{conjecture8.1} without constraint on $s$, it 
follows that
\begin{align*}
\Ome_1(X)&\ll X^{k(k-1)/2+\eps}\left( X^s\text{mes}(\grD_1)+X^{2s-k(k+1)/2}\right) \\
&\ll X^\eps (X^s+X^{2s-k}).
\end{align*}
We therefore conclude from (\ref{8.9}) that
\[
X^{2s(1-1/k)-(k-1)/2}\ll \Ome_1(X)\ll X^\eps(X^s+X^{2s-k}).
\]
This is tenable only when
\[
2s(1-1/k)-(k-1)/2\le \max\{s,2s-k\},
\]
which is to say that either $s\le \tfrac{1}{2}(k+1)+1/(k-2)$, or $2s\ge k(k+1)/2$. Thus we 
find that the upper bound asserted in Conjecture \ref{conjecture8.1} cannot hold in general 
in the absence of a condition at least as strong as $s\ge \tfrac{1}{4}k(k+1)$.

\bibliographystyle{amsbracket}
\providecommand{\bysame}{\leavevmode\hbox to3em{\hrulefill}\thinspace}

\end{document}